\documentclass[a4paper,10pt]{article}
\usepackage{amsmath,amsthm,xypic,
amssymb,amsfonts,
pdfsync,stmaryrd,mathrsfs,yfonts
}
\newtheorem{thm}{Theorem}[section]

\newtheorem{prop}[thm]{Proposition}
\newtheorem{cor}[thm]{Corollary}

\newtheorem{lem}[thm]{Lemma}
\theoremstyle{definition}

\newtheorem{defi}[thm]{Definition}
\newtheorem{conj}[thm]{Conjecture}

\newtheorem{rem}[thm]{Remark}

\theoremstyle{plain}

\newcommand{\lla}{\langle\!\langle}
\newcommand{\rra}{\rangle\!\rangle}

\renewcommand{\phi}{\varphi}

\newcommand{\alt}{\mathop{\mathrm{Alt}}}

\newcommand{\nrd}{\mathop{\mathrm{Nrd}}}

\newcommand{\End}{\mathop{\mathrm{End}}}

\newcommand{\id}{\mathop{\mathrm{id}}}
\newcommand{\ind}{\mathop{\mathrm{ind}}}

\newcommand{\disc}{\mathop{\mathrm{disc}}}

\newcommand{\ad}{\mathrm{Ad} }

\author{A.-H. Nokhodkar}

\title{Pfister involutions in characteristic two}
\date{}

\begin{document}
\maketitle

\begin{abstract}
In characteristic two, it is shown that a central simple algebra of degree equal to a power of two with anisotropic ortho\-gonal involution is totally decomposable, if it becomes either anisotropic or metabolic over all extensions of the ground field.
A similar result is obtained for the case where this
algebra with involution is Brauer-equivalent to a quaternion algebra and it becomes adjoint to a bilinear Pfister form over all splitting fields of the algebra.
\\


\end{abstract}

\section{Introduction}
An algebra with involution $(A,\sigma)$ is called {\it totally decomposable} if it decomposes as a tensor product of quaternion algebras with involution.
Every split totally decomposable algebra with involution is adjoint to a bilinear Pfister form.
This is known as the Pfister Factor Conjecture, formulated by D. B. Shapiro \cite{shapiro} and proved by K. J. Becher \cite{becher} (see \cite{mn} for a proof in characteristic two).
A more general conjecture was stated in \cite{bayer} as follows: for an algebra with involution $(A,\sigma)$ of degree $2^n$ over a field $F$ of characteristic not two the following statements are equivalent:
(1) $(A,\sigma)$ is totally decomposable.
(2) For every splitting field $K$ of $A$, $(A,\sigma)_K$ is adjoint to a Pfister form.
(3) For every field extension $K/F$, $(A,\sigma)_K$ is either anisotropic or hyperbolic.
The implications $(1)\Rightarrow (2)$ and $(1)\Rightarrow (3)$ and the equivalence $(2)\Leftrightarrow(3)$ are already proved in the literature (see \cite[Theorem 1]{becher}, \cite[(1.1)]{kar} and \cite[(3.2)]{black}).
However, the implication $(2)\Rightarrow (1)$ or $(3)\Rightarrow (1)$ is still open in general and is solved only for certain special cases (see for example \cite[(2.10)]{bayer} and \cite[Theorem 2]{becher}).

A similar conjecture may be considered for orthogonal involutions in characteristic two.
Since an orthogonal involution in characteristic two can not be hyperbolic, one may replace the hyperbolicity condition with the metabolicity.
Also, using \cite[(5.5)]{dolphin} one can find a metabolic bilinear form $\mathfrak{b}$ of dimension $2^n$ over a field $F$ which is not similar to any Pfister form.
This can be used to show that the implication $(3)\Rightarrow(2)$ is not true in characteristic two.
However, by \cite[(5.5)]{lagh} an anisotropic symmetric bilinear form  of dimension $2^n$ over $F$, which is either anisotropic or metabolic over all extensions of $F$, is similar to a Pfister form.
Considering this result a conjecture may be formulated as follows (see \cite{dolphin}):
\begin{conj}\label{con}
Let $(A,\sigma)$ be a central simple algebra of degree $2^n$ with orthogonal involution over a field $F$ of characteristic two.
If $\sigma$ is anisotropic, then the following statements are equivalent.
\begin{itemize}
  \item [(1)] $(A,\sigma)$ is totally decomposable.
  \item [(2)] For every splitting field $K$ of $A$, $(A,\sigma)_K$ is adjoint to bilinear Pfister form.
  \item [(3)] For every field extension $L/F$, $(A,\sigma)_L$ is either anisotropic or metabolic.
\end{itemize}
\end{conj}
The implication $(1)\Rightarrow(2)$ follows from \cite[(3.6)]{mn} and $(1)\Rightarrow(3)$ was proved in \cite[(6.2)]{dolphin}.
The aim of this work is to study some other implications of this conjecture.
First, we prove Conjecture \ref{con} for the case where $A$  is Brauer-equivalent to a quaternion algebra.
This result is an analogue of \cite[Theorem 2]{becher}, stated in characteristic different from $2$.
We then prove the implication $(3)\Rightarrow(1)$ (for arbitrary $A$) in Theorem \ref{main2}.
We finally consider this conjecture for direct involutions defined in \cite{dolphin2} as follows:
an involution $\sigma$ on $A$ is called {\it direct} if $\sigma(a)a\in\alt(A,\sigma)$ implies that $a=0$ for $a\in A$.
It is clear that every direct involution is anisotropic.
According to \cite[(6.1)]{dolphin}, the converse is also true for totally decomposable orthogonal involutions in characteristic two.
Our last result shows that Conjecture \ref{con} is true in the case where $\sigma$ is direct (see Theorem \ref{main3}).
\section{Preliminaries}
Throughout this work, $F$ denote a field of characteristic $2$.

Let $A$ be a finite-dimensional central simple algebra over $F$.
The square root of $\dim_FA$ is called the {\it degree} of $A$ and is denoted by $\deg_F A$.
If $B$ is a subalgebra of $A$, the centralizer of $B$ in $A$ is denoted by $C_A(B)$.
According to the Wedderburn's theorem, $A$ is isomorphic to the matrix algebra $M_n(D)$ for some division $F$-algebra $D$.
The algebra $D$ is called the {\it division algebra component} of $A$.
Also, the integer $n$ is called the {\it co-index} of $A$ and is denoted by $\mathfrak{i}(A)$.
In other words, $\mathfrak{i}(A)=\frac{\deg_FA}{\ind A}$, where $\ind A$ is the Schur index of $A$.
A {\it quaternion algebra} over $F$ is a central simple $F$-algebra of degree $2$.
Every quaternion algebra has a basis $(1,u,v,w)$, called a {\it quaternion basis}, satisfying $u^2+u\in F$, $v^2\in F^\times$ and $w=uv=vu+v$ (see \cite[p. 25]{knus}).

By an {\it involution} on $A$ we mean an antiautomorphism of $A$ of order $2$.
An involution is said to be of {\it the first kind}, if it leaves $F$ elementwise invariant.
Involutions of the first kind are either {\it symplectic} or {\it orthogonal} (see \cite[(2.5)]{knus}).
An involution $\sigma$ on $A$ is called {\it isotropic} if $\sigma(a)a=0$ for some nonzero element $a\in A$.
Otherwise, $\sigma$ is called {\it anisotropic}.
We say that $\sigma$ is {\it metabolic} if there exists an idempotent $e\in A$ such that $\sigma(e)e=0$ and $\dim_FeA=\frac{1}{2}\dim_FA$.
Metabolic involutions were introduced first in \cite{berhuy1}.

If $(A,\sigma)$ is an algebra with involution over $F$ and $K/F$ is a field extension, we denote $(A,\sigma)\otimes_F(K,\id)$ by $(A,\sigma)_K$.
We also use the notation
$\alt(A,\sigma)=\{x-\sigma(x)\mid x\in A\}$.
If $\sigma$ is orthogonal and $A$ is of even degree, the {\it discriminant} of $\sigma$ is defined as $\disc\sigma=\nrd_A(x)F^{\times}/F^{\times2}$, where $x\in\alt(A,\sigma)$ is a unit and $\nrd_A(x)$ is the reduced norm of $x$ in $A$.
Finally, for a symmetric bilinear space  $(V,\mathfrak{b})$ over $F$, the pair $(\End_F(V),\sigma_\mathfrak{b})$ is denoted by $\ad(\mathfrak{b})$, where $\sigma_\mathfrak{b}$ is the adjoint involution of $\End_F(V)$ with respect to $\mathfrak{b}$ (see \cite[p. 2]{knus}).

\section{Split factors}
\begin{lem}\label{desort}
Let $K/F$ be a finite field extension satisfying $K^2\subseteq F$ and let $(Q,\sigma)$ be a quaternion algebra with orthogonal involution over $K$.
If $v^2\in F^\times$ for some $v\in\alt(Q,\sigma)$, then there exists a quaternion $F$-subalgebra $Q_0$ of $Q$ such that $\sigma(Q_0)=Q_0$.
In addition, if $Q$ splits, then $Q_0$ can be chosen to be split.
\end{lem}

\begin{proof}
The first statement follows from \cite[(6.1)]{mn2}.
To prove the second one let $\alpha=v^2\in F^\times$, so that $\disc\sigma=\alpha F^{\times2}$.
Let $\langle1,\alpha\rangle$ be the diagonal bilinear form $\mathfrak{b}((x_1,x_2),(y_1,y_2))=x_1y_1+\alpha x_2y_2$ over $F$.
Then $(A,\sigma)\simeq\ad(\langle1,\alpha\rangle_K)$ by \cite[(7.3 (3)) and (7.4)]{knus}.
The result therefore follows from the isomorphism $\ad(\langle1,\alpha\rangle_K)\simeq\ad(\langle1,\alpha\rangle)_K$.
\end{proof}

\begin{defi}
Let $(A,\sigma)$ be a totally decomposable algebra with involution over $F$.
We say that $(A,\sigma)$ has a decomposition with $s$ {\it split factors}
if there exists a decomposition $(A,\sigma)\simeq\bigotimes_{i=1}^{n}(Q_i,\sigma_i)$ such that $Q_i\simeq M_2(F)$ for $i\leqslant s$ and $Q_i$ is a quaternion division algebra for $i>s$.
\end{defi}
Note that if $(A,\sigma)$ has a decomposition with $s$ split factors, then $\mathfrak{i}(A)$ is a multiple of $2^s$.

The next result follows from \cite[\S9, Theorem 12 (c)]{draxl}.

\begin{lem}\label{dr}
Let $A$ be a central simple algebra over $F$.
If $K/F$ is a finite field  extension, then $\mathfrak{i}(A_K)=[K:F]\cdot\mathfrak{i}(A)$ if and only if $K$ can be embedded into the division algebra component of $A$.
\end{lem}

Let $(A,\sigma)\simeq\bigotimes_{i=1}^n(Q_i,\sigma_i)$ be a totally decomposable algebra with ortho\-gonal involution over $F$.
As observed in \cite{mn2}, there exists a unique, up to isomor\-phism, subalgebra $\Phi(A,\sigma)\subseteq F+\alt(A,\sigma)$ of dimension $2^n$ such that (i) $C_A(\Phi(A,\sigma))=\Phi(A,\sigma)$; (ii) $\Phi(A,\sigma)$ is generated, as an $F$-algebra, by $n$ elements; and (iii) $x^2\in F$ for every $x\in\Phi(A,\sigma)$ (see \cite[(4.6), (5.10) and (5.11)]{mn2}).
Note that if $v_i\in\alt(Q_i,\sigma_i)$ is a unit for $i=1,\cdots,n$, then $\Phi(A,\sigma)\simeq F[v_1,\cdots,v_n]$.

\begin{prop}\label{main}
Let $(A,\sigma)$ be a totally decomposable algebra with orthogonal involution over $F$ and let $\mathfrak{i}(A)=2^s$ for some nonnegative integer $s$.
Then $(A,\sigma)$ has a decomposition with $s$ split factors if and only if
$\Phi(A,\sigma)$ has a subfield $L$ containing $F$ with $[L:F]=2^{n-s}$ which splits $A$.
\end{prop}

\begin{proof}
Suppose first that $(A,\sigma)$ has a decomposition $(A,\sigma)\simeq\bigotimes_{i=1}^n(Q_i,\sigma_i)$ such that $Q_i$ splits for $i=1,\cdots s$.
As $\mathfrak{i}(A)=2^s$, the division algebra component of $A$ is isomorphic to $\bigotimes_{i=s+1}^n Q_i$.
Choose a unit $w_i\in\alt(Q_i,\sigma_i)$, $i=1,\cdots,n$.
Then $\Phi(A,\sigma)\simeq F[w_1,\cdots,w_n]$ and one can take $L=F[w_{s+1},\cdots,w_n]$.

Conversely, suppose that there exists a subfield $L\subseteq\Phi(A,\sigma)$ with $[L:F]=2^{n-s}$ which splits $A$.
If $s=n$ (i.e., $A$ splits and $L=F$), the result follows from \cite[(3.6)]{mn}.
Suppose that $s<n$, so $L\neq F$.
We use induction on $n$.
The case $n=1$ is evident, hence let $n\geqslant2$.
Since $A_L$ splits, we have $\mathfrak{i}(A_L)=2^n=[L:F]\cdot\mathfrak{i}(A)$.
By Lemma \ref{dr} there exists an embedding $L\hookrightarrow D$, where $D$ is the division algebra component of $A$.
Choose an element $w\in L\setminus F$ and set $K=F[w]$ and $B=C_A(w)$.
Since $w\in \Phi(A,\sigma)$ and $w^2\notin F^{\times2}$, by \cite[(6.3 (i))]{mn2}, $(B,\sigma|_B)$ is a totally decomposable algebra with orthogonal involution over $K$ and $\Phi(B,\sigma|_B)\simeq\Phi(A,\sigma)$ as $K$-algebras.
Hence, we may identify $L$ with a subfield of $\Phi(B,\sigma|_B)$.
As $K\hookrightarrow D$, using Lemma \ref{dr} we get $\mathfrak{i}(A_K)=2^{s+1}$.
It follows that $\mathfrak{i}(B)=2^s$, because $B\otimes M_2(K)\simeq_K A_K$.
Note that $L$ splits $B$, $[L:K]=2^{n-s-1}$ and $\deg_KB=2^{n-1}$.
Thus, by induction hypothesis, $(B,\sigma|_B)$  has a decomposition with $s$ split factors.
Let
\[\textstyle(B,\sigma|_B)\simeq_K\bigotimes_{i=1}^{n-1}(Q_i,\sigma_i),\]
be such a decomposition into quaternion $K$-algebras with involution for which $Q_i$ splits, $i=1,\cdots,s$.
Choose a unit $v_i\in\alt(Q_i,\sigma_i)$, $i=1,\cdots,n-1$.
Then $\Phi(B,\sigma|_B)\simeq K[v_1,\cdots,v_{n-1}]$.
Since $\Phi(B,\sigma|_B)\simeq\Phi(A,\sigma)$, we have $v_i^{2}\in F^\times$ for every $i$.
By Lemma \ref{desort}, there exists a quaternion algebra with involution $(Q'_i,\sigma'_i)$ over $F$ such that
$(Q_i,\sigma_i)\simeq(Q'_i,\sigma'_i)_K$, $i=1,\cdots,n-1$.
Also, for $i=1,\cdots,s$, $Q'_i$ can be chosen to be split.
The algebra $\bigotimes_{i=1}^{n-1}Q'_i$ may be identified with a subalgebra of $A$.
Set $Q'_n=C_A(\bigotimes_{i=1}^{n-1}Q'_i)$ and $\sigma'_n=\sigma|_{Q'_n}$.
Then $(Q'_n,\sigma'_n)$ is a quaternion $F$-algebra with involution and
\[\textstyle(A,\sigma)\simeq\bigotimes_{i=1}^n(Q'_i,\sigma'_i).\]
Thus, $(A,\sigma)$ has a decomposition with $s$ split factors, proving the result.
\end{proof}

\begin{lem}\label{K2}
Let $K/F$ be a finite field extension with $K^2\subseteq F$ and let $Q$ be a quaternion algebra over $F$.
If $Q_K$ splits, then there exists $x\in Q\setminus F$ such that $x^2\in K^2$.
\end{lem}

\begin{proof}
Let $(1,u,v,w)$ be a quaternion basis of $Q$.
Set $a=u^2+u\in F$ and $b=v^2\in F^\times$.
Since $Q_K$ splits, by \cite[(98.14) (5)]{elman}, $b$ is a norm in the quadratic \'etale extension $K_a:=K[X]/(X^2+X+a)$ of $K$, i.e.,
$b=c^2+cd+d^2a$  for some $c,d\in K$.
If $d=0$, then $b=c^2\in K^2$ and we are done.
Suppose that $d\neq 0$.
Set $e:=cd^{-1}\in K$ and $x=ev+w\in Q\setminus F$.
Then $b=d^2(e^2+e+a)$ and
\[x^2=e^2b+eb+ab=b(e^2+e+a)=d^2(e^2+e+a)^2\in K^2.\qedhere\]
\end{proof}

\begin{cor}\label{bi}
Let $(A,\sigma)$ be a totally decomposable algebra of degree $2^n$ with ortho\-gonal involution over $F$.
If $\mathfrak{i}(A)=2^{n-1}$, then $(A,\sigma)$ has a decomposition with $n-1$ split factors.
\end{cor}

\begin{proof}
Let $L$ be a maximal subfield of $\Phi(A,\sigma)$ containing $F$.
Since $\mathfrak{i}(A)=2^{n-1}$, the division algebra component $D$ of $A$ is a quaternion algebra.
By \cite[(5.9)]{mn2}, $L$ is a splitting field of $A$, hence $D_L$ splits.
By Lemma \ref{K2} there exists $v\in D\setminus F$ such that $v^2\in L^2$.
It follows that $L\subseteq\Phi(A,\sigma)$ has a subfield isomorphic to $F[v]\subseteq D$ which splits $D$ (and therefore $A$).
Since $L^2\subseteq F$ we have $[F[v]:F]=2$.
Hence the result follows from Proposition \ref{main}.
\end{proof}

\section{Orthogonal Pfister involutins}
Let $(A,\sigma)\simeq\bigotimes_{i=1}^n(Q_i,\sigma_i)$ be a totally decomposable algebra with ortho\-gonal involution over $F$ and let $\alpha_i\in F^\times$ be a representative of the class $\disc\sigma_i\in F^\times/F^{\times2}$, $i=1,\cdots,n$.
By \cite[(7.3)]{dolphin}, the bilinear Pfister form $\lla\alpha_1,\cdots,\alpha_n\rra$ is independent of the decomposition of $(A,\sigma)$.
As in \cite{dolphin}, we denote this form by $\mathfrak{Pf}(A,\sigma)$.

The next result follows from \cite[(4.6 (iii)), (5.5) and (5.6)]{mn2}.

\begin{lem}\label{lem}
Let $(A,\sigma)$ be a totally decomposable algebra with orthogonal invo\-lution over $F$.
If $\mathfrak{Pf}(A,\sigma)\simeq\lla\alpha_1,\cdots,\alpha_n\rra$ for some $\alpha_1,\cdots,\alpha_n\in F^\times$, then there exist $v_1,\cdots,v_n\in\alt(A,\sigma)$ such that $\Phi(A,\sigma)= F[v_1,\cdots,v_n]$ and $v_i^2=\alpha_i$ for $i=1,\cdots,n$.
\end{lem}

\begin{prop}\label{prop}
Let $(A,\sigma)$ be a central simple algebra of degree $2^n$ with anisotropic orthogonal involution over $F$ and let $K/F$ be a separable quadratic extension.
Then $(A,\sigma)$ is totally decomposable if and only if $(A,\sigma)_K$ is totally decomposable.
\end{prop}

\begin{proof}
Suppose that $(A,\sigma)_K$ is totally decomposable.
By \cite[(6.7)]{me}, there exists a symmetric bilinear form $\mathfrak{b}$ over $F$ such that $\mathfrak{Pf}((A,\sigma)_K)\simeq\mathfrak{b}_K$.
Since $\mathfrak{b}_K$ is a Pfister form, the form $\mathfrak{b}$ is similar to a Pfister form $\mathfrak{c}$ over $F$ by \cite[(3.3)]{dolphin}.
Hence, $\mathfrak{c}_K$ is similar to $\mathfrak{Pf}(A,\sigma)$.
As $\mathfrak{c}_K$ represents $1$, we get $\mathfrak{c}_K\simeq\mathfrak{Pf}(A,\sigma)$ (see \cite[(2.4)]{dolphin}).
Write $\mathfrak{c}=\lla\alpha_1,\cdots,\alpha_n\rra$ for some $\alpha_1,\cdots,\alpha_n\in F^\times$.
Then $\mathfrak{Pf}(A,\sigma)\simeq\lla\alpha_1,\cdots,\alpha_n\rra_K$.
By Lemma \ref{lem} one can write $\Phi((A,\sigma)_K)=K[v_1,\cdots,v_n]$, where $v_i\in\alt((A,\sigma)_K)$ and $v_i^2=\alpha_i$, $i=1,\cdots,n$.
Choose $\eta\in K$ with $\delta:=\eta^2+\eta\in F$ for which $K=F(\eta)$.
Since $\alt((A,\sigma)_K)=\alt(A,\sigma)\otimes_FK$, every $v_i$ can be written as $v_i=u_i\otimes1+w_i\otimes\eta$, where $u_i,w_i\in\alt(A,\sigma)$.
Hence
\[v_i^2=(u_i^2+\delta w_i^2)\otimes1+(u_iw_i+w_iu_i+w_i^2)\otimes\eta.\]
As $v_i^2\in F^\times$, we have $u_iw_i+w_iu_i+w_i^2=0$.
It follows that
\begin{align}\label{eqnew}
  w_i^2=u_iw_i-\sigma(u_iw_i)\in\alt(A,\sigma).
\end{align}
If $\sigma_K$ is isotropic, then it is metabolic by \cite[(6.2)]{dolphin}, which contradicts \cite[(5.6)]{dol}.
Thus, $\sigma_K$ is anisotropic.
Since $(A,\sigma)_K$ is totally decomposable, \cite[(6.1)]{dolphin} implies that $\sigma_K$ is direct.
By (\ref{eqnew}) we have
\[\sigma_K(w_i\otimes1)\cdot(w_i\otimes1)=(w_i^2\otimes1)\in\alt(A,\sigma)_K.\]
Hence, $w_i=0$, i.e., $v_i=u_i\otimes1$ for $i=1,\cdots,n$.
Set $S=F[u_1,\cdots,u_n]\subseteq A$, so that $\Phi((A,\sigma)_K)=S\otimes_FK$.
Note that $S\subseteq F+\alt(A,\sigma)$ is an $F$-algebra of dimension $2^n$, $C_A(S)=S$ and $x^2\in F$ for every $x\in S$.
Hence, $(A,\sigma)$ is totally decomposable by \cite[(4.6)]{mn2}, proving the `if' implication.
The converse is trivial.
\end{proof}

The next result gives a solution to Conjecture \ref{con} for the case where $A$ is Brauer-equivalent to a quaternion algebra.
\begin{cor}{\rm (Compare \cite[Theorem 2]{becher})}\label{main2}
Let $(A,\sigma)$ be a central simple $F$-algebra of degree $2^n$ with anisotropic orthogonal involution.
If $A$ is Brauer-equivalent to a quaternion algebra $Q$ over $F$, then the following statements are equivalent.
\begin{itemize}
  \item [(1)] $(A,\sigma)$ is totally decomposable.
  \item [(2)] If $K$ is a splitting field of $A$, then $(A,\sigma)_K$ is adjoint to bilinear Pfister form.
  \item [(3)]$(A,\sigma)\simeq(Q,\tau)\otimes\ad(\mathfrak{b})$ for some orthogonal involution $\tau$ on $Q$ and some $(n-1)$-fold bilinear Pfister form $\mathfrak{b}$ over $F$.
  \item [(4)] For every field extension $L/F$, $(A,\sigma)_L$ is either anisotropic or metabolic.
\end{itemize}
\end{cor}

\begin{proof}
The implication $(3)\Rightarrow (1)$ is evident.
The implications $(1)\Rightarrow (2)$, $(1)\Rightarrow (3)$ and
$(1)\Rightarrow (4)$ follow from \cite[(3.6)]{mn}, Corollary \ref{bi} and \cite[(6.2)]{dolphin} respectively (even without the anisotropy condition on $\sigma$).
Let $K\subseteq Q$ be a separable quadratic extension of $F$.
Then $K$ is a splitting field of $A$.
In view of Proposition \ref{prop}, to prove $(2)\Rightarrow (1)$ and $(4)\Rightarrow (1)$ it is enough to show that $(A,\sigma)_K$ is totally decomposable.
This is clear if the condition $(2)$ is satisfied.
Otherwise, $(A,\sigma)_K$ is totally decomposable by \cite[(8.3)]{dolphin}.
\end{proof}
The following result is the analogue of Merkurjev's theorem in characteristic two, which was proved by O. Teichm\"uller (see \cite[(9.1.4)]{gille}).
\begin{thm}\label{mer}
Every central simple $F$-algebra of exponent $2$ is Brauer-equivalent to a tensor product of quaternion algebras.
\end{thm}

\begin{rem}\label{sep}
Let $A$ be a central simple algebra of exponent two over $F$.
Then there exists a chain $K_0\subseteq K_1\subseteq\cdots\subseteq K_m$ of fields with $K_0=F$ such that $K_m$ splits $A$ and every $K_i/K_{i-1}$ is a quadratic separable extension.
Indeed, by Theorem \ref{mer}, $A$ is Brauer equivalent to a tensor product of quaternion $F$-algebras $\bigotimes_{i=1}^nQ_i$.
Hence, the required chain can be constructed inductively as follows:
set $K_0=F$ and suppose that $K_j$ is constructed.
If $K_j$ splits all $Q_i$, then it splits $A$ and we are done.
Otherwise, let $r$ be the minimal index for which $K_j$ does not split $Q_r$.
Then $Q_r\otimes_FK_j$ is a quaternion division $K_j$-algebra and $K_{j+1}\hookrightarrow Q_r\otimes_FK_j$ can be chosen as any quadratic separable extension of $K_j$.
\end{rem}

We are now ready to prove the implication $(3)\Rightarrow(1)$ in conjecture \ref{con}.
\begin{thm}\label{main3}
For a central simple algebra of degree $2^n$ with anisotropic ortho\-gonal involution $(A,\sigma)$ over $F$, the following statements are equivalent.
\begin{itemize}
  \item [(1)] $(A,\sigma)$ is totally decomposable.
  \item [(2)] For every field extension $K/F$, $(A,\sigma)_K$ is either anisotropic or metabolic.
\end{itemize}
\end{thm}

\begin{proof}
The implication $(1)\Rightarrow (2)$ follows from \cite[(6.2)]{dolphin}.
To prove the converse, observe that by \cite[(3.1)]{knus}, the exponent of $A$ is at most $2$.
 Hence, by Remark \ref{sep} there exists a chain $F\subseteq K_1\subseteq\cdots\subseteq K_m$ of fields such that every $K_i/K_{i-1}$ is a quadratic separable extension and $K_m$ splits $A$.
By \cite[(8.3)]{dolphin}, the pair $(A,\sigma)_{K_m}$ is totally decomposable.
We claim that $(A,\sigma)_{K_i}$ is anisotropic for $i=1,\cdots,m$.
Suppose the contrary; let $r\geqslant1$ be the minimal index for which $(A,\sigma)_{K_r}$ is isotropic.
The assumption implies that $(A,\sigma)_{K_r}$ is metabolic.
Hence, $(A,\sigma)_{K_{r-1}}$ is isotropic by \cite[(5.6)]{dolphin}, contradicting the minimality of $r$.
The claim is therefore proved.
By Proposition \ref{prop} and induction on $m$, the pair $(A,\sigma)$ is totally decomposable.
\end{proof}

\begin{rem}
Let $(A,\sigma)$ be a central simple $F$-algebra of degree $2^n$ with anisotropic orthogonal involution.
Suppose that for all field extensions $K/F$, $(A,\sigma)_K$ is either anisotropic or metabolic.
In \cite[(8.3)]{dolphin}, it is shown that for every field extension $L/F$ such that $A_L$ splits, there exists a bilinear Pfister form $\mathfrak{b}$ over $L$ such that $(A,\sigma)_L\simeq\ad(\mathfrak{b})$.
It is also asked whether there exists a bilinear Pfister form $\mathfrak{b}'$ over $F$ such that $\mathfrak{b}\simeq\mathfrak{b}'_L$ (see \cite[(8.4)]{dolphin}).
Using Theorem \ref{main3} one can find an affirmative answer to this question.
Indeed, by Theorem \ref{main3}, $(A,\sigma)$ is totally decomposable, hence $(A,\sigma)_L\simeq\ad(\mathfrak{Pf}(A,\sigma)_L)$ by \cite[(7.5 (3))]{dolphin}.
It follows from \cite[(4.2)]{knus} that $\mathfrak{b}$ is similar to $\mathfrak{Pf}(A,\sigma)_L$.
Since $\mathfrak{Pf}(A,\sigma)_L$ and $\mathfrak{b}$ are both Pfister forms, these forms are actually  isometric.
\end{rem}

We conclude by proving Conjecture \ref{con} for direct involutions.
\begin{thm}\label{main2}
Let $(A,\sigma)$ be a central simple $F$-algebra of degree $2^n$ with orthogonal involution.
If $\sigma$ is direct, then the following statements are equivalent.
\begin{itemize}
  \item [(1)] $(A,\sigma)$ is totally decomposable.
  \item [(2)] If $K$ is a splitting field of $A$, then $(A,\sigma)_K$ is adjoint to bilinear Pfister form.
  \item [(3)] For every field extension $L/F$, $(A,\sigma)_L$ is either anisotropic or metabolic.
\end{itemize}
\end{thm}

\begin{proof}
Since $\sigma$ is direct, it is anisotropic.
Hence, the implications $(1)\Rightarrow(2)$, $(1)\Rightarrow(3)$ and $(3)\Rightarrow(1)$ follow from \cite[(3.6)]{mn}, \cite[(6.2)]{dolphin} and Theorem \ref{main3}.
To prove $(2)\Rightarrow(1)$ we proceed similar to the proof of Theorem \ref{main3}.
Using Remark \ref{sep}, one can choose a chain $F\subseteq K_1\subseteq\cdots\subseteq K_m$ of fields such that every $K_i/K_{i-1}$ is a quadratic separable extension and $K_m$ splits $A$.
By \cite[(9.1)]{dolphin2} (and induction on $m$), $(A,\sigma)_{K_i}$ is anisotropic for $i=1,\cdots,m$.
Since $(A,\sigma)_{K_m}$ is totally decomposable by \cite[(8.3)]{dolphin}, the conclusion follows from Proposition \ref{prop} and induction on $m$.
\end{proof}

\scriptsize{
\noindent
A.-H. Nokhodkar, {\tt
  a.nokhodkar@kashanu.ac.ir},\\
Department of Pure Mathematics, Faculty of Science, University of Kashan, P.~O. Box 87317-51167, Kashan, Iran.}

\end{document}